\documentclass{amsart} 
\usepackage{amsmath}
\usepackage{amssymb}
\usepackage{amsthm}

\newtheorem{defn0}{Definition}[section]
\newtheorem{prop0}[defn0]{Proposition}
\newtheorem{thm0}[defn0]{Theorem}
\newtheorem{lemma0}[defn0]{Lemma}
\newtheorem{corollary0}[defn0]{Corollary}
\newtheorem{example0}[defn0]{Example}
\newtheorem{remark0}[defn0]{Remark}
\newtheorem{conjecture0}[defn0]{Conjecture}

\newenvironment{proposition}{\bigskip \begin{prop0}}{\end{prop0}}
\newenvironment{theorem}{\bigskip \begin{thm0}}{\end{thm0}}
\newenvironment{lemma}{\bigskip \begin{lemma0}}{\end{lemma0}}
\newenvironment{corollary}{\bigskip \begin{corollary0}}{\end{corollary0}}

\def\cocoa
{\mbox{\rm C\kern-.13em o\kern-.07
em C\kern-.13em o\kern-.15em A}}

\newcommand{\m}{\mathfrak{m}}

\newcommand{\la}{\lambda}
\newcommand{\q}{\mathfrak{q}}
\newcommand{\M}{\mathbb{M}}
\newcommand{\N}{\mathbb{N}}

\begin{document}
\title{ On the  Chern number of a filtration }
\author{  M. E. Rossi and G. Valla}
\address{Dipartimento di Matematica, Universit\'a di Genova, Via Dodecaneso 35, I-16146 Genova, Italy }
\email{ rossim@dima.unige.it, valla@dima.unige.it}
\date{}
\thanks{Research partially supported 
by the ``Ministero dell'Universit\'a e della  Ricerca Scientifica"   
in the framework of the National Research Network (PRIN 2005) 
``Algebra commutativa, combinatorica e computazionale"}

\date{\today}


\maketitle


\bigskip
\section{Introduction.} The Hilbert function of a local ring $(A,\m)$ is  the function which associates to each  non negative integer $n$ the minimal number of generators of $\m^n.$ It  gives deep information on the corresponding singularity and, owing to its geometric significance, has been studied extensively. 

In the case of a standard graded algebra, the Hilbert function is well understood, at least in the Cohen-Macaulay case. Instead, very little is known in the local case and we do not have  a guess of the shape of the Hilbert function, even  in the case of a Cohen-Macalay one dimensional domain.

Due to  this lack of information, a long list of papers have been written where  constraints on the possible Hilbert functions of a Cohen-Macaulay local ring are described.  Usually,  restrictions have been found on what are called the Hilbert coefficients of the maximal ideal $\m$ of $A$.


The first coefficient, $e_0(\m)$, is called the multiplicity of $\m$ and, due to its  geometric meaning, has been studied very deeply, see for example  \cite{S}. The other coefficients are not as well understood, either geometrically or in terms of how they are related to algebraic properties of the ideal or ring.

The goal of this paper is the study of the second   (after the multiplicity)  coefficient of the Hilbert polynomial  of the local ring $(A,\m).$  This integer,  $e_1,$  has been recently interpreted   in \cite{PUV} as a {\bf tracking} number of the Rees algebra of $A$ in the set of all such algebras with the same multiplicity. Under various circumstances,  it is also called the {\bf Chern number or coefficient} of the local ring $A.$ 
An interesting list of questions and conjectural
statements about the values of $e_1$  for filtrations
associated to an  $\m$-primary ideal  of a local  ring $A$ had been presented in a recent paper by  W. Vasconcelos (see \cite{V1}).

\vskip 2mm
In the Cohen-Macaulay case, starting from the work of D.G. Northcott in the 60's, several results have been proved which give some relationships  between the Hilbert coefficients, thus implying some constrains on the Hilbert function. The way was enlightened by Northcott who proved that, if $(A,\m)$ is a Cohen-Macaulay local ring, then $$e_1(\m)\ge e_0(\m)-1.$$ Thus, for example, the series $\frac{1+2z-z^2}{1-z}$  cannot be the Hilbert series of a Cohen-Macaulay local ring of dimension one because $e_1(\m)=0,$ while $e_0(\m)=2.$.

Several  results along  this line has been proved in the last years, often extending the framework to Hilbert function associated to a filtration, but always assuming the Cohen-Macaulayness of the basic ring. For a large overview of these kind of results one can see  \cite{RV}.

More recently Goto and Nishida in \cite{GN} were able  to extend  Northcott bound to the case of a primary ideal,  avoiding,  for the first time in this setting, the assumption that the ring is Cohen-Macayulay. Of course, they need to introduce a  correction term which vanishes in the Cohen-Macaulay case. 

On the base of this result,  Corso in  \cite{C} could  handle in this wild generality  a stronger upper bound for $e_1$ given by Elias and Valla in \cite{EV}.

Here we will focus, instead, on an upper bound of $e_1$ which was introduced and studied by Huckaba and Marley in \cite{HM}. If $\q$ is a primary ideal in the local Cohen-Macaulay ring $A$ and   $J$ is an ideal generated by a maximal superficial sequence for $\q$, then Valla proved in \cite{V} that $$e_0(\q)=(1-d)\la(A/\q)+\la(\q/\q^2)+\la(\q^2/J\q),$$ thus extending a classical result of Abhyankar  on the maximal ideal to the $\m$-primary case. The above formula shows that $\la(\q^2/J\q)$ does not depend on $J$ and, for  the first time,  properties of the Hilbert coefficients are related to  superficial sequences. Since then,  it was  natural to consider the integers  $$v_j(\q):=\la(\q^{j+1}/J\q^j)$$ and to investigate  how they are involved  in the study of the Hilbert series of a primary ideal. This has been done by Huckaba and Marley in \cite{HM} where they proved that for any ideal filtration $\mathcal{F}=\{I_j\}_{j\ge 0}$ of a Cohen-Macaulay local ring $A,$ one has $$e_1(\mathcal{F})\le \sum_{j\ge 0}v_j(\mathcal{F}),$$ with equality if and only if the depth of the associated graded ring is at least $d-1,$ if and only if the Hilbert series is $$P_{\mathcal{F}}(z)=\frac {\la(A/I_0)+\sum_{j\ge 0}(v_j(\mathcal{F})-v_{j+1}(\mathcal{F}))z^{j+1}}{(1-z)^d}.$$

The first main result of this paper, Theorem \ref{th1}, shows that given any  good $\q$-filtration  $\mathbb{M}$ of a module $M$  which is not necessarily Cohen-Macaulay, we have $$e_1(\mathbb{M})-e_1(\mathbb{N})\le \sum_{j\ge 0}v_j(\mathbb{M}),$$ where $\mathbb{N}$ is the $J$-adic filtration on $M$. It is easy to see that, if $M$ is Cohen-Macaulay, then $e_i(\mathbb{N})=0$ for every $i\ge 1,$ so that these integers are good correction terms when the Cohen-Macaulay assumption does not hold.

We are able to understand in Theorem \ref{th2} when equality holds above, at least for the case when $\mathbb M$ is the $\m$-adic filtration on the local ring $A.$ Surprisingly enough, we show that the equality $e_1(\m)-e_1(\mathbb{N})= \sum_{j\ge 0}v_j(\m)$ forces the ring $A$ to be Cohen-Macaulay and hence the associated graded ring to have almost maximal depth.

Thus we get the Cohen-Macaulayness of the ring $A$ as a consequence of the extremal behavior of the integer $e_1(\m).$  The result can be considered a confirm of the general philosophy of  the paper of W. Vasconcelos \cite{V1} where the Chern number is conjectured to be a measure of the distance from the   Cohen-Macaulyness of $A.$ 

This main result of the paper is a consequence of a nice and perhaps unexpected property of superficial elements which we prove in Lemma \ref{nuovo}. It is essentially a kind of  `` Sally machine" for local rings.

In the last section we describe an application of these results, concerning an upper  bound on the multiplicity of the Sally module  of a good filtration of a module which is not necessarily Cohen-Macaulay. The bound is given in term of the same ingredients used  by Huckaba-Marley for $e_1.$ It is an extension  to the non Cohen-Macaulay case of a result of Vaz Pinto in \cite{Vaz}.

\section{An  upper bound for the Chern number}
Let $(A,\m)$ be  a local ring and $\q$ a primary ideal for $\m.$ If $M\neq\{0\}$ is a finitely generated $A$-module, a $\q$-filtration on $M$ is a chain $$\mathbb{M}:  \ \   M=M_0\supseteq M_1 \supseteq \dots \supseteq M_j\supseteq ...$$ such that $\q M_j\subseteq M_{j+1}$ for every $j.$  In the case $M_{j+1}=\q M_j$ for $j\gg 0,$ we say that the filtration $\mathbb{M}$  is a {\bf good} $\q$-filtration.

If $N$ is a submodule of $M$, it is clear that $\{(N+M_j)/N\}_{j\ge 0}$ is a good $\q$-filtration of $M/N$ which  we denote by $\mathbb{M}/N.$

Given the good $\q$-filtration $\mathbb{M}$ on $M$ we let $gr_{\q}(A):=\oplus_{j\ge 0}\q^j/\q^{j+1},$  $gr_{\mathbb{M}}(M):=\oplus_{j\ge 0}M_j/M_{j+1}.$ We know that $gr_{\mathbb{M}}(M)$ is a graded $gr_{\q}(A)$-module; further each element $a\in A$ has a natural image $a^*$ in $gr_{\q}(A)$ which is $0$ if $a=0,$ is $a^*=\overline{a}\in \q^t/\q^{t+1}$ if $a\in \q^t\setminus \q^{t+1}.$

A tool which has been very useful in the study of the depth of blowung-up rings is the so called Valla-Valabrega criterion (see \cite{VV}). Given the ideal $I=(a_1,\dots,a_r)$ in $A$,  the elements $a_1^*,\dots,a_r^*$ form  a regular sequence on  $gr_{\mathbb{M}}(M)$ if and only if they form a regular sequence on $M$ and $IM\cap M_j=IM_{j-1}$ for every $j\ge 1.$ 

We recall that an element $a\in \q$ is called $\mathbb{M}$-superficial for $\q$ if there exists a non-negative integer $c$ such that $(M_{j+1}:_Ma)\cap M_c=M_j$ for every $j\ge c.$ If we assume that $M$ has positive dimension, then every superficial element $a$ for $\q$ has order one, that is $a\in \q \setminus \q^2.$
Further, since we are assuming that the residue field $A/\m$ is infinite, it is well known that superficial elements do exist. 

A sequence of elements $a_1,\dots,a_r$ will be called a $\mathbb{M}$-superficial sequence for $\q$ if for every $j=1,\dots,r$ the element $a_j$ is an $(\mathbb{M}/(a_1,\dots,a_{j-1})M)$-superficial element  for $\q.$ It is well known that if $a_1,\dots,a_r$ is a $\mathbb{M}$-superficial sequence   for $\q$ and $J$ is the ideal they generate, then we have the following  properties which show the relevance of superficial sequences in the study of the depth of blowing-up rings.

\begin{equation} \{a_1,\dots,a_r\} \ \  \text{is a regular sequence on} \  M \  \text{if and only if} \ \  depth\  M \ge r.\end{equation}

\begin{equation} \{a_1^*,\dots,a_r^*\} \  \text{is a regular sequence on} \ \ gr_{\mathbb{M}}(M)\  \text{if and only if}\  depth~ gr_{\mathbb{M}}(M) \ge r.\end{equation}

\vskip 2mm \noindent If the sequence is maximal then, as
in \cite{Sw},  \begin{equation}\label{Jgood} M_{n+1}=JM_n~ \text{for} ~n\gg 0.\end{equation} Hence $\M$ is a good $J$-filtration. Further by 8.3.5 in \cite{HS} , $J$ is minimal with respect to this property.
It follows that if $x_1,\dots,x_d\in J$ is a maximal $\M$-superficial sequence, then $J=(x_1,\dots,x_d).$

\vskip 3mm We recall now that the Hilbert function of the filtration $\mathbb{M}$ is the function $$H_{\mathbb{M}}(j):=\la(M_j/M_{j+1}).$$ The Hilbert series of $\mathbb{M}$ is  $$P_{\mathbb{M}}(z):=\sum_{j\ge 0}H_{\mathbb{M}}(j)z^j=\frac{h_{\mathbb{M}}(z)}{(1-z)^d}$$ where $d$ is the dimension of $M.$ For every $i\ge 0$ we let  $e_i(\mathbb{M}):=\frac{h_{\mathbb{M}}^{(i)}(1)}{i!};$ then for every $j\gg 0$ we have $ H_{\mathbb{M}}(j)=p_{\mathbb{M}}(j)$ where 

$$p_{\mathbb{M}}(X):=\sum_{i=0}^{d-1}(-1)^ie_i(\mathbb{M})\binom{X+d-i-1}{d-i-1}$$ is a polynomial with rational coefficients and degree $d-1$ which  is called  the Hilbert polynomial of $\mathbb{M}.$ Its coefficients $e_i(\mathbb{M})$ are the {\bf Hilbert coefficients} of $\mathbb{M}.$

It is also relevant to consider the numerical function $$H_{\mathbb{M}}^1(j):=\la(M/M_{j+1})=\sum_{i=0}^jH_{\mathbb{M}}(i)$$ which is called the Hilbert-Samuel function of the filtered module $M.$  Its generating function is simply the series $$P^1_{\mathbb{M}}(z)=\frac{
P_{\mathbb{M}}(z)}{(1-z)}=\frac{h_{\mathbb{M}}(z)}{(1-z)^{d+1}}.$$ It is clear that the polynomial $$p_{\mathbb{M}}^1(X):=\sum_{i=0}^d(-1)^ie_i(\mathbb{M})\binom{X+d-i}{d-i}$$ verifies the equality $p_{\mathbb{M}}^1(j)=H_{\mathbb{M}}^1(j)$ for $j\gg 0.$
It is called the Hilbert-Samuel polynomial of $\mathbb{M}.$

The first Hilbert coefficient $e_0$  is the multiplicity of  $\mathbb{M}$; by Proposition 11.4 in \cite{BH} we know that \begin{equation}\label{e=e} e_0(\mathbb{M})=e_0(\mathbb{N})\end{equation} for every couple of good $\q$-filtrations. Also, if $M$ is artinian, then $e_0(\mathbb{M})=\la(M).$

The classical case is when the module $M$ is the ring $A$ and the filtration is given by the powers of a primary  ideal $I$ of $A$. This is called the $I$-adic filtration on $A$ and  the corresponding Hilbert coefficients are  simply indicated by $e_i(I).$ 

For every $a\in \q$ one can prove that \begin{equation}\label{singh} P_{\mathbb{M}}(z)\le P^1_{\mathbb{M}/aM}(z)\end{equation}(see  \cite{RV} and  \cite{Singh}).

As a consequence of (\ref{singh}),  we get some useful properties of superficial elements. In the following $a$ will be  a $\mathbb{M}$-superficial element for $\q$ and $d$ the dimension of $M.$ We have 

   \begin{equation}\label{d} \dim(M/aM)=d-1.
    \end{equation}
 \begin{equation}\label{ej} e_j(\mathbb{M})=e_j(\mathbb{M}/aM)  \ \ \ \text {for every}\ \ \ j=0,\dots,d-2.\end{equation}
\begin{equation}\label{ed}e_{d-1}(\mathbb{M}/aM)=e_{d-1}(\mathbb{M})+(-1)^{d-1}\la(0:_Ma). \end{equation}

\vskip 3mm We will also need a property of superficial elements which seems to be neglected in the literature. Recall that if $a$ is $M$-regular, then $M/aM$ is Cohen-Macaulay if and only if $M$ is Cohen-Macaulay. We prove,  as a consequence of  the following Lemma, that the same holds for a superficial element,  if the dimension of the module $M$ is at least two.  

In the following we denote by $H_{\q}^i(M)$ the $i$-th local cohomology module  of $M$ with respect to $\q.$ We know that $H_{\q}^0(M):=\cup_{j\ge 0}(0:_M\q^j)=0:_M\q^t$ for every $t\gg 0$ and $\min\{i \ | \ H^i_{\q}(M)\neq 0\}=\text{grade}(\q,M),$ where $\text{grade}(\q,M)$ is  the common cardinality of all the maximal $M$-regular sequences of elements in $\q.$
\begin{lemma} \label{nuovo} Let $\mathbb{M}$ be a good $\q$-filtration of  a module $M,$  $a$ 
an $\mathbb{M}$-superficial element for $\q$ and $j\ge 1$.Then  $\text{grade}(\q,M)\ge j+1$ if and only if  $grade(\q,M/aM)\ge j.$
\end{lemma}
\begin{proof} Let  $\text{grade}(\q,M)\ge j+1$; then  depth $M>0$ so that  $a$ is $M$-regular. This implies   $\text{grade}(\q,M/aM)= \text{grade}(\q,M)-1\ge j+1-1=j.$

Let us assume now that  $\text{grade}(\q,M/aM)\ge j.$ Since $j\ge 1,$ this implies  $H^0_{\q}(M/aM)=0.$ Hence $H^0_{\q}(aM)=H^0_{\q}(M),$ so that $H^0_{\q}(M)\subseteq aM.$ We claim that $H^0_{\q}(M)=aH^0_{\q}(M).$ If this is the case, then,  by Nakayama, we get $H^0_{\q}(M)=0$ which  implies  depth $M>0,$ so that   $a$ is $M$-regular.  Hence 

$$\text{grade}(\q,M)=\text{grade}(\q,M/aM)+1\ge j+1,$$ 
as wanted. 

Let us prove the claim. Suppose by contradiction that $$aH^0_{\q}(M)\subsetneq H^0_{\q}(M)\subseteq aM,$$ and let $ax\in H^0_{\q}(M), x\in M\setminus H^0_{\q}(M).$ This means that for every  $t\gg 0$ we have $$\begin{cases} a\q^tx=0\\ \q^tx\neq 0\end{cases}.$$ We prove  that this implies that $a$ is not
 $\mathbb{M}$-superficial  for $\q.$  Namely, given a positive integer $c,$ we can find an integer   $t\ge c$  and an element $d\in \q^t$ such that $adx=0$ and $dx\neq 0.$ Since $ \cap M_i=\{0\},$ we have $dx\in M_{j-1}\setminus M_j$ for some integer $j.$ Now, $d\in \q^t$ hence $dx\in M_t\subseteq M_c,$ which implies  $j\ge c.$  Finally we have $dx\in M_c,$ $dx\notin M_j$ and $adx=0 \in M_{j+1},$ hence $$(M_{j+1}:_Ma)\cap M_c\supsetneq M_j.$$ The claim and the Lemma are proved.\end{proof}

\vskip 2mm In establishing the properties of the Hilbert coefficients of a filtered module $M$, it will be convenient to use induction on the dimension of the module. To start the induction we need first to study the one-dimensional case.

Let us be given  a good $\q$-filtration $\mathbb{M}=\{M_j\}_{j\ge 0}$ on a module $M$ of dimension 1. We have $H_{\mathbb{M}}(n)=p_{\mathbb{M}}(n)=e_0(\mathbb{M})$ for $n\gg 0$ so that we define for every $j$ \begin{equation}\label{uij} u_j(\mathbb{M}):=e_0(\mathbb{M})-H_{\mathbb{M}}(j).
\end{equation}
\begin{lemma}\label{uij} Let $\mathbb{M}$ be a good $\q$-filtration of a module $M$ of dimension one. If $a$ is an $\mathbb{M}$-superficial element for $\q,$ then for every $j\ge 0$ we have $$u_j(\mathbb{M})=\la(M_{j+1}/aM_j)-\la(0:_{M_j}a).$$
\end{lemma}
\begin{proof} By (\ref{d}) we have \begin{equation*}\begin{split} e_0(\mathbb{M})& =e_0(\mathbb{M}/aM)-\la(0:_Ma)=\la(M/aM)-\la(0:_Ma)\\ &= \la(M/aM_j)-\la(aM/aM_j)-\la(0:_Ma).\\
\end{split}\end{equation*}

By using the following exact sequence $$0\to (0:_Ma)/(0:_{M_j}a)\to M/M_j\to  aM/aM_j\to 0$$ we get 
\begin{equation*}\begin{split}
e_0(\mathbb{M})&=\la(M/aM_j)-\la(M/M_j)+\la((0:_Ma)/(0:_{M_j}a))-\la(0:_Ma)\\
\end{split}\end{equation*} and finally 
\begin{equation*}\begin{split} u_j(\mathbb{M})&=e_0(\mathbb{M})-\la(M_j/M_{j+1})\\ &
= \la(M/aM_j)-\la(M/M_j)-\la(0:_{M_j}a)-\la(M_j/M_{j+1})\\
&=\la(M_{j+1}/aM_j)-\la(0:_{M_j}a).\\ \end{split} \end{equation*} \end{proof}

 It follows that, in the case $M$  is one dimensional and Cohen-Macaulay, then $u_j(\mathbb{M})=\la(M_{j+1}/aM_j)$ is  non negative and we have, for every $j\ge 0$, $H_{\mathbb{M}}(j)=e_0(\mathbb{M})-\la(M_{j+1}/aM_j)\le e_0(\mathbb{M}).$

It will be useful to write down the Hilbert coefficients through the integers $ u_j(\mathbb{M}).$ 

\begin{lemma}\label{eu} Let $\mathbb{M}$ be a good $\q$-filtration of a module $M$ of dimension one. Then for every $j\ge 0$ we have $$e_j(\mathbb{M})=\sum_{k\ge j-1}\binom{k}{j-1}u_k(\mathbb{M}).$$
\end{lemma}
\begin{proof} We have  $$P_{\mathbb{M}}(z)=\frac{h_{\mathbb{M}}(z)}{1-z}=\sum_{j\ge 0}H_{\mathbb{M}}(j)z^j.$$ Hence, if we write $h_{\mathbb{M}}(z)=h_0(\mathbb{M})+h_1(\mathbb{M})z+\dots +h_s(\mathbb{M})z^s,$ then we get for every $k\ge 1$ $$h_k(\mathbb{M})=H_{\mathbb{M}}(k)-H_{\mathbb{M}}(k-1)=u_{k-1}(\mathbb{M})-u_k(\mathbb{M}).$$ Finally \begin{equation*}\begin{split}e_j(\mathbb{M})&=\frac{h_{\mathbb{M}}^{(j)}(1)}{j!}=\sum_{k\ge j}\binom{k}{j}h_k(\mathbb{M}) =\sum_{k\ge j}\binom{k}{j}(u_{k-1}(\mathbb{M})-u_k(\mathbb{M}))\\ & =\sum_{k\ge j-1}\binom{k}{j-1}u_k(\mathbb{M}).\\
\end{split} \end{equation*}
\end{proof}

If we apply the above Lemma in the case $M$ is Cohen-Macaulay, and using the fact that the integers $u_k(\mathbb{M})$ are non negative,  we get $$e_1(\mathbb{M})=\sum_{k\ge 0}u_k(\mathbb{M})\ge u_0(\mathbb{M})+u_1(\mathbb{M})\ge u_0(\mathbb{M}).$$ Since we have $u_0(\mathbb{M})=e_0(\mathbb{M})-\la(M/M_1),$ and $u_0(\mathbb{M})+u_1(\mathbb{M})=2e_0(\mathbb{M})-\la(M/M_2),$ we trivially  get 
$$e_1(\mathbb{M})\ge e_0(\mathbb{M})-\la(M/M_1),$$ and $$e_1(\mathbb{M})\ge 2e_0(\mathbb{M})-\la(M/M_2),$$ which are the bounds   of Northcott and Elias-Valla, in the one dimensional Cohen-Macaulay case.

If we do not assume that $M$ is Cohen-Macaulay, the integers $u_k(\mathbb{M})$ can be negative and the above formulas do not hold anymore.

At this point we are going to describe the correction term which will appear in our upper bound of $e_1.$

Given a good $\q$-filtration $\mathbb{M}$ of the module $M$ of dimension $d,$ let $a_1,\dots,a_d$ be a   $\mathbb{M}$-superficial sequence for $\q.$ We denote by $J$ the ideal they generate and consider the $J$-adic filtration of the module $M$. This is by definition the filtration $$\mathbb{N}:=\{J^jM\}_{j\ge 0}$$ which is clearly a good $J$-filtration.  By (\ref{Jgood}) $\M$ is also a good $J$-filtration,  so that, by (\ref{e=e}),  $e_0(\mathbb{M})=e_0(\mathbb{N}).$ 

In the case $M$ is Cohen-Macaulay, the elements $a_1,\dots,a_d$ form a regular sequence on $M$ so that $J^iM/J^{i+1}M\simeq (M/JM)^{\binom{d+i-1}{i}}.$ This implies that  the Hilbert Series of $\mathbb{N}$ is $P_{\mathbb{N}}(z)=\frac{\la(M/JM)}{(1-z)^d}$ and thus   $e_i(\mathbb{N})=0$ for every $i\ge 1.$ This proves that   these integers  give a good measure of how $M$ differs from being  Cohen-Macaulay.
In the case $M=A, $ Vasconcelos in \cite{Vas} and \cite{V1} conjectured that if $A$ is not Cohen-Macaulay, then $e_1(J) <0 $  and he proved it for large classes of local rings. 
\vskip 2mm

First we prove that $ e_1(\mathbb{N}) \le  0$  in the one dimensional case where the integer $e_1(\mathbb{N})$ can be easily related with the 0-th local cohomology module of $M.$

\begin{lemma}\label{e1N} Let $\mathbb{M}$ be a good $\q$-filtration of a module $M$ of dimension one. If $a$ is an $\mathbb{M}$-superficial element for $\q$, then for every $t\gg 0$ we have $$e_1(\mathbb{N})=-\la(0:_Ma^t).$$\end{lemma}
\begin{proof} We have $p_{\mathbb{N}}^1(X)=e_0(\mathbb{N})(X+1)-e_1(\mathbb{N})=e_0(\mathbb{M})(X+1)-e_1(\mathbb{N}).$ On the other hand we have short exact sequences $$0\to (0:_Ma^k)/0:_{aM}a^k\to M/aM\overset{a}\to a^kM/a^{k+1}M \to 0$$ 
$$0\to 0:_Ma\to 0:_Ma^{i+1}\overset{a}\to 0:_{aM}a^i\to 0$$ which give

\begin{equation*}
\begin{split} &H^1_{\mathbb{N}}(j)=\sum_{i=0}^jH_{\mathbb{N}}(i)=\sum_{i=0}^j\la(a^iM/a^{i+1}M)=\sum_{i=0}^j\left [ \la(M/aM)-\la (0:_Ma^i/ 0:_{aM}a^i)\right ]\\
&=(j+1)\la(M/aM)-\la ( 0:_Ma)-\sum_{i=1}^j\left [  \la(0:_Ma^{i+1})-\la(0:_{aM}a^i)\right ]+\la(0:_Ma^{j+1})\\
&=(j+1)[\la(M/aM)-\la ( 0:_Ma)]+\la(0:_Ma^{j+1})=(j+1)e_0(\mathbb{M})+\la(0:_Ma^{j+1}).
\end{split}
\end{equation*} The conclusion follows.

\end{proof}

In this section we denote by $W$ the $0$-th local cohomology module $H_{\m}^0(M)$ of $M$ with respect to $\m.$ Recall that $$W:=\cup_{j\ge 0}(0:_M\m^j).$$ When $M$ is a module of dimension one and $a$  an $\mathbb{M}$-superficial element for $\q,$  then $M/aM$ has finite length so that $0:_A(M/aM)$ is a primary ideal. Now it is easy to see that $0:_A(M/aM)\subseteq \sqrt{(a)+0:_AM},$ and thus $(a)+0:_AM$ is a primary ideal too. This means that a power of $\m$ is contained in $(a)+0:_AM$, say $\m^s\subseteq (a)+0:_AM.$ Hence, for $t\gg 0$ $$W=0:_M\m^t\subseteq 0:_Ma^t\subseteq 0:_M \m^{ts}=W.$$ The above lemma and this remark  give  \begin{equation}\label{w} \la(W)=-e_1(\mathbb{N})\end{equation} as was already proved in \cite{GN}, 
Lemma 2.4.

Given a good $\q$-filtration of the module $M$, we consider now the corresponding filtration of the saturated module $M^{sat}:=M/W.$ This is the filtration $$\mathbb{M}^{sat}:=\mathbb{M}/W=\{M_n+W/W\}_{n\ge 0.}$$ Since $W$ has finite length and $\cap M_i=\{0\},$ we have $M_i\cap W=\{0\}$ for every $i\gg 0.$ This implies  $p_{\mathbb{M}}(X)=
p_{\mathbb{M}^{sat}}(X).$

\vskip 2mm Further, it is clear that for every $j\ge 0$ we have an exact sequence $$ 0\to W/(M_{j+1}\cap W)\to M/M_{j+1}\to M/(M_{j+1}+W)\to 0$$ so that for every $j\gg 0$ we have $$\la ( M/M_{j+1})=\la [M/(M_{j+1}+W)]+\la(W)$$ which implies \begin{equation} \label {sat} p^1_{\mathbb{M}}(X)=p^1_{\mathbb{M}^{sat}}(X)+\la(W).\end{equation} This proves the following result:
\begin{proposition}\label{locoh} Let $\mathbb{M}$ be a good $\q$-filtration of the module $M$  and $W:=H^0_{\m}(M).$ If we let $d:=\dim(M)$ and $\mathbb{M}^{sat}:=\mathbb{M}/W,$ then  $$ e_i(\mathbb{M})=e_i({\mathbb{M}}^{sat}) \ \    0\le i \le d-1,\ \ \ \ \ \ \ \ \
e_d(\mathbb{M})=e_d({\mathbb{M}}^{sat})+(-1)^d\la(W).$$
\end{proposition}

We remark that, if $\dim(M)\ge 1$, the module $M/W$ has always positive depth. This is the reason why, sometimes,  we move our attention from the module $M$ to the module $M/W.$ This will be  the strategy of the proof of  the next proposition which gives,  in the one dimensional case, the promised upper bound for $e_1.$

\begin{proposition}\label{d=1} Let $\mathbb{M}$ be a good $\q$-filtration of a module $M$ of dimension one. If $a$ is a $\mathbb{M}$-superficial element for $\q$ and $\mathbb{N}$ the $(a)$-adic filtration on $M,$ then
$$e_1(\mathbb{M})-e_1(\mathbb{N})\le \sum_{j\ge 0}\la(M_{j+1}/aM_j).$$  If  $W\subseteq M_1$ and  equality holds above, then  $M$ is Cohen-Macaulay.\end{proposition}
\begin{proof} By Proposition \ref{locoh} and (\ref{w}) we have $$e_1(\mathbb{M})=e_1(\mathbb{M}^{sat})-\la(W)=e_1(\mathbb{M}^{sat})+e_1(\mathbb{N})$$ so that we need to prove that $e_1(\mathbb{M}^{sat})\le \sum_{j\ge 0}\la(M_{j+1}/aM_j).$

Now $M/W$ is Cohen-Macaulay and $a$ is regular on $M/W,$ hence by Lemma \ref{eu} and Lemma \ref{uij}, we get 
\begin{equation*}
\begin{split} 
e_1(\mathbb{M}^{sat})&= 
\sum_{j\ge 0}u_j(\mathbb{M}^{sat})=\sum_{j\ge 0}\la(M_{j+1}^{sat}/aM_j^{sat})\\
&=\sum_{j\ge 0}\la\left[\frac{M_{j+1}+W}{aM_j+W}\right]
=\sum_{j\ge 0}\la\left[\frac{M_{j+1}}{aM_j+M_{j+1}\cap W}\right]  \\
& \le     \sum_{j\ge 0}\la(M_{j+1}/aM_j).
\end{split}
\end{equation*} The first assertion  follows.  In particular equality holds if and only if $M_{j+1}\cap W\subseteq aM_j$ for every $j\ge 0.$ Let as assume $W\subseteq M_1$ and equality above; then we have $W=W\cap M_1\subseteq aM.$ Now recall that $W=0:_Ma^t$ for $t\gg 0,$ hence if $c\in W$ then $c=am$ with $a^tc=a^{t+1}m=0.$ This implies  $m\in 0:_Ma^{t+1}=W$ so that $W\subseteq aW$ and, by Nakayama, $W=0.$
\end{proof}

In the last section we need to compare the Hilbert coefficients of the $J$-adic filtration $\N$ on $M$ with those of the following filtration. Given a good $\q$-filtration $\M$ on $M$ and the ideal $J$ generated by a maximal $\M$-superficial sequence for $\q,$ we let $$\mathbb{E}(J,\M):= M \supseteq M_1 \supseteq JM_1\supseteq J^2M_1\supseteq\dots \supseteq J^{n-1}M_1\supseteq\dots $$ When there is no ambiguity we simply write $\mathbb{E}$ instead of $\mathbb{E}(J,\M).$

It is clear that this is a good $J$-filtration  so that $$e_0(\M) =e_0(\mathbb{E})=e_0(\N).$$
By considering the leading coefficients of  $ p^1_{\M}(X) - p^1_{\mathbb {E}}(X)$ and 
$p^1_{\mathbb {E}}(X) - p^1_{\N}(X),   $  since $ J^{n+1}M \subseteq J^n M_1 \subseteq M_{n+1} $ it follows that  $$e_1(\M) \ge e_1(\mathbb{E})\ge e_1(\N).$$

\begin{proposition} \label{EN} Let $\M$ be a good $\q$-filtration  of the module $M;$  if $\dim M=1$ and $J=(a)$ with a $\M$-superficial for $\q$, then $$ e_1(\mathbb{E})-e_1(\N)\ge e_0(\M)-h_0(\M).$$
\end{proposition}
\begin{proof}  It is clear that for every $n\gg 0$ we have $$p^1_{\mathbb {E}}(n)=\la(M/a^nM_1)=e_0(\M)(n+1)-e_1(\mathbb{E}),$$
$$p^1_{\N}(n-1)=\la(M/a^nM)=e_0(\M)n-e_1(\N).$$ It follows that $$e_1(\mathbb{E})-e_1(\N)=e_0(\M)+\la(M/a^nM)-\la(M/a^nM_1)\ge e_0(\M)-\la(a^nM/a^nM_1).$$ The conclusion follows because  the map $M/M_1\overset{a^n}\to a^nM/a^nM_1$ is surjective.
\end{proof}

We come now to the higher dimensional case; a natural strategy will be to lower dimension by using superficial elements.

The following Lemma is the key for doing the job. It is due to David Conti.

\begin{lemma}\label{key} Let $M_1,\dots,M_r$ be $A$-modules of dimension $d,$  $J$ a $\m$-primary ideal of $A$ and $\M_1=\{M_{1,j}\},\dots,\M_r=\{M_{r,j}\}$ good $J$-filtrations of $M_1,\dots,M_r$ respectively. Then we can find elements $a_1,\dots,a_d$ which are $\M_i$-superficial for $J$ for every $i=1,\dots,r.$

If $d\ge 2$ and  $M_1=\dots=M_r=M,$  then  for every $1\le i\le j\le r$ we have $$e_1(\M_i)-e_1(\M_i/a_1M)=e_1(\M_j)-e_1(\M_j/a_1M).$$
\end{lemma}
\begin{proof} Let 
$$\oplus^r_{i=1}\M_i:=\ \  \oplus_{i=1}^rM_i\supseteq \oplus_{i=1}^rM_{i,1}\supseteq \oplus_{i=1}^rM_{i,2}\supseteq  \dots \supseteq \oplus_{i=1}^rM_{i,j}\supseteq  \dots$$ It is clear that this is a good $J$-filtration on $\oplus_{i=1}^rM_i.$ Let us choose a $\oplus^r_{i=1}\M_i$-superficial  sequence
$\{a_1,\dots,a_d\}$ for $J$. Then it is easy to see that $\{a_1,\dots,a_d\}$ is a sequence of $M_i$-superficial  elements
for $J$ for every $i=1,\dots,r.$ This proves the first assertion. As for the second one,  by (\ref{ej}) and  (\ref{ed}) we have 
$$e_1(\M_i)-e_1(\M_i/a_1M)=\begin{cases} \la(0:_Ma_1) & \text{if}\ \ d=2\\
0 & \text{if} \ \ d\ge 3\end{cases}$$ from which the conclusion follows.
\end{proof}

We first extend to the higher dimensional case the result of  Proposition \ref{EN}.  We will see that the following result  is a strengthened version of  the classical inequality due to Northcott.  
\begin{proposition}\label{EN1}  Let $\M$ be a good $\q$-filtration  of the module $M$  and let $J$ be the ideal generated by a maximal sequence of $\M$-superficial elements for $\q.$ Then we have  $$ e_1(\mathbb{E})-e_1(\N)\ge e_0(\M)-h_0(\M).$$
\end{proposition}
\begin{proof} If $\dim M=1$ we use Proposition \ref{EN}. Let $\dim M\ge 2;$ by the above Lemma we can find elements $a_1,\dots,a_d\in J$  which are superficial for $J$ with respect to $\mathbb{E}$ and $\N$ at the same time. Further $$e_1(\mathbb{E})-e_1(\mathbb{E}/a_1M)=e_1(\N)-e_1(\N/a_1M)$$ and $J=(a_1,\dots,a_d)$ by (\ref{Jgood}).

The module $M/a_1M$ has dimension $d-1$ and $\M/a_1M$ is a good $J$-filtration on it. Further $a_2,\dots,a_d$ is a maximal $\M/a_1M$-superficial sequence for $J.$  If we let $K:=(a_2,\dots,a_d),$ then the $K$-adic filtration on $M/a_1M$ is given by $$\{K^j(M/a_1M)\}_{j\ge 0}=\{(K^jM+a_1M)/a_1M\}_{j\ge 0}=\{(J^jM+a_1M)/a_1M\}_{j\ge 0}$$ and thus coincides with $\N/a_1M.$ On the other hand the filtration
$$\mathbb{E}(K,\M/a_1M):= M/a_1M\supseteq \frac{M_1+a_1M}{a_1M}\supseteq K\frac{M_1+a_1M}{a_1M}\supseteq K^2\frac{M_1+a_1M}{a_1M}\supseteq... $$ coincides with the filtration $\mathbb{E}/a_1M$ because $K^j\frac{M_1+a_1M}{a_1M}=\frac{J^jM_1+a_1M}{a_1M}.$ By   induction we get 
$$e_1(\mathbb{E})-e_1(\N)=e_1(\mathbb{E}/a_1M)-e_1(\N/a_1M)\ge e_0(\M/a_1M)-h_0(\M/a_1M)=e_0(\M)-h_0(\M)$$ as desired.
\end{proof}
\vskip 3mm
Since $e_1(\M)  \ge e_1(\mathbb{E})$ we obtain Nothcott's inequality which had been extended to   non Cohen-Macaulay case by  Goto and Nishida in \cite{GN}.

\begin{corollary}\label{north}  Let $\M$ be a good $\q$-filtration  of the module $M$  and let $J$ be the ideal generated by a maximal sequence of $\M$-superficial elements for $\q.$ Then we have  $$ e_1(\M)-e_1(\N)\ge e_0(\M)-h_0(\M).$$
\end{corollary}
\vskip 3mm
Given a good $\q$-filtration $\M=\{M_j\}_{j\ge 0}$ of the $A$-module  $M$ and an ideal $J$ generated by a maximal sequence of $\M$-superficial elements for $\q$, we let for every $j\ge 0:$
$$v_j(\M):=\la(M_{j+1}/JM_j).$$
\begin{theorem}\label{th1} Let $\M$ be a good $\q$-filtration of a module $M$ of dimension $d\ge 1,$  $J$ an ideal generated by a maximal sequence of $\M$-superficial elements for $\q$ and $\N$ the $J$-adic filtration of $M.$ Then we have $$e_1(\M)-e_1(\N)\le \sum_{j\ge 0}v_j(\M).$$
\end{theorem}
\begin{proof} We prove the Theorem by induction on $d.$ If $d=1$ we can apply Proposition \ref{d=1}; hence let $d\ge 2$. As in the above Proposition, we can find a system of generators  $a_1,\dots,a_d$    of $J$ with the properties  that  are superficial for $J$ with respect to $\M$ and $\N,$ and verify  $$e_1(\N)-e_1(\N/a_1M)=e_1(\M)-e_1(\M/a_1M)$$  by (\ref{Jgood}). Also, if we let $K:=(a_2,\dots,a_d),$ then $K$ is 
generated by a maximal $\M/a_1M$-superficial sequence for $J$ and the $K$-adic filtration on $M/a_1M$ is $\N/a_1M.$ Thus, by induction, we get
\begin{equation*}
\begin{split} e_1(\M)-e_1(\N)&=e_1(\M/a_1M)-e_1(\N/a_1M)\le \sum_{j\ge 0} v_j(\M/a_1M) \\ 
&=\sum_{j\ge 0}\la\left[(M_{j+1}+a_1M)/(KM_j+a_1M)\right]\\
&=\sum_{j\ge 0}\la\left[(M_{j+1}+a_1M)/(JM_j+a_1M)\right]\\
&= \sum_{j\ge 0}\la \left [ M_{j+1}/(JM_j+(a_1M\cap M_{j+1}))\right ] \\
&\le \sum_{j\ge 0} \la(M_{j+1}/JM_j).
\end{split}\end{equation*}
\end{proof}

We would like to study when the equality in the above Theorem holds. In the case $M$ is Cohen-Macaulay, we have a solution for  a general filtration which broadly  extends previous results of \cite{HM} and \cite{GR}.

\begin{theorem}\label {e1=} Let $\mathbb{M}=\{M_j\}_{j\ge 0}$ be a good $\q$-filtration of the Cohen-Macaulay module $M$ of dimension $d\ge 1$ and    $J$  an ideal generated by a maximal sequence of $\mathbb{M}$-superficial  elements 
for $\q.$ Then we have 
\vskip 2mm 
a)  $e_1(\mathbb{M})\le \sum_{j\ge 0}v_j(\mathbb{M})$ 
\vskip 2mm 
b) $e_2(\mathbb{M})\le \sum_{j\ge 0}jv_j(\mathbb{M}).$
\vskip 2mm 
c) The following conditions are equivalent
\vskip 2mm
\noindent
1. $\text{depth} \ gr_{\mathbb{M}}(M)\ge d-1.$

\vskip 1mm
\noindent
2. $e_i(\mathbb{M})=\sum_{j\ge i-1}\binom{j}{i-1}v_j(\mathbb{M})$ for every $i\ge 1.$

\vskip 1mm 
\noindent
3. $e_1(\mathbb{M})=\sum_{j\ge 0}v_j(\mathbb{M}).$

\vskip 1mm 
\noindent
4. $e_2(\mathbb{M})=\sum_{j\ge 0}jv_j(\mathbb{M}).$

\vskip 1mm 
\noindent
5. $P_{\M}(z)=\frac{\la(M/M_1)+\sum_{j\ge 0}(v_j(\M)-v_{j+1}(\M))z^{j+1}}{(1-z)^d}.$

  \end{theorem}
   \begin{proof} Let $J=(a_1,\cdots,a_d)$ and $K=(a_1,\cdots,a_{d-1});$ we first remark that, by   Valabrega-Valla,  $\text{depth} \ gr_{\mathbb{M}}(M)\ge d-1$ if and only if $M_{j+1}\cap KM  = KM_j$ for every $j\ge 0.$ Further  we have $$v_j(\mathbb{M})=v_j(\mathbb{M}/KM)+\la(M_{j+1}\cap KM+JM_j/JM_j),$$ hence $v_j(\mathbb{M})\ge v_j(\mathbb{M}/KM)$ and equality holds if and only if 
   $M_{j+1}\cap KM   \subseteq JM_j.$ This  is certainly the case when $M_{j+1}\cap KM=KM_j;$ hence, if $\text{depth} \ gr_{\mathbb{M}}(M)\ge d-1,$ then $v_j(\mathbb{M})=v_j(\mathbb{M}/KM)$ for every $j\ge 0.$ By induction on $j,$ we can prove that the converse holds. Namely $M_1 \cap KM=KM$ and, if $j\ge 1,$  then we have
   \begin{equation*}\begin{split} M_{j+1}\cap KM & \subseteq JM_j\cap KM=(KM_j+a_dM_j)\cap KM\\
   &=KM_j+(a_dM_j\cap KM)\subseteq KM_j+a_d(M_j\cap KM)\\
   &=KM_j+a_dKM_{j-1}=KM_j 
   \end{split}
   \end{equation*}  where $a_dM_j\cap KM\subseteq a_d(M_j\cap KM)$ because $a_d$ is regular modulo $KM,$ while  $M_j\cap KM=KM_{j-1}$ follows by induction.
   
   Since $M/KM$ is Cohen-Macaulay of dimension one, we get
   $$e_1(\mathbb{M})=e_1(\mathbb{M}/KM)=\sum_{j\ge 0}v_j(\mathbb{M}/KM)\le \sum_{j\ge 0}v_j(\mathbb{M}).$$ Equality holds if and only if   $\text{depth} \ gr_{\mathbb{M}}(M)\ge d-1.$ This, once more, proves a)  and moreover gives the equivalence between 1 and 3 in c). By using  Lemma \ref{uij} this also gives the equivalence between 1. and 5. in c).
   
   Further, if $\frak{a}$ is the ideal generated by $a_1,\cdots,a_{d-2},$ then, as before, we get
   $$e_2(\mathbb{M})=e_2(\mathbb{M}/\frak{a}M)\le e_2(\mathbb{M}/KM) =\sum_{j\ge 1}jv_j(\mathbb{M}/KM)\le \sum_{j\ge 1}jv_j(\mathbb{M}).$$ This proves b) and 4 $\Longrightarrow$1. To complete the proof of the Theorem, we need only to show that
   1 $\Longrightarrow$ 2. If $\text{depth} \ gr_{\mathbb{M}}(M)\ge d-1,$ then $\mathbb{M}$ and $\mathbb{M}/KM$ have the same $h$-polynomial; this implies that for every $i\ge  1$ we have $$e_i(\mathbb{M})=e_i(\mathbb{M}/KM)=\sum_{j\ge i-1}\binom{j}{i-1}v_j(\mathbb{M}/KM)=\sum_{j\ge i-1}\binom{j}{i-1}v_j(\mathbb{M}).$$
   \end{proof}

If we do not assume that $M$ is Cohen-Macaulay, we are able to handle the problem only for the $\m$-adic filtration on $A.$

\begin{theorem}\label{th2} Let $(A,\m)$ be a local ring of dimension $d\ge 1$  and $J$ the ideal generated by a maximal  $\m$-superficial sequence. The following conditions are equivalent:

\vskip 2mm 
1. $e_1(\m)-e_1(J)=\sum_{j\ge 0}v_j(\m).$

\vskip 2mm
2. $A$ is Cohen-Macaulay and depth $gr_{\m}(A)\ge d-1.$
\end{theorem}
\begin{proof} If $A$ is Cohen-Macaulay, then $e_1(J)=0$ and,  by the above result, we get that 2) implies 1).

We prove now that 1) implies 2) by induction on $d.$ If $d=1,$ the result follows by Proposition \ref{d=1} since $W\subseteq \m.$ Let $d\ge 2$; by Lemma \ref{key} we can find a minimal basis $\{a_1,\dots,a_d\}$ of $J$ such that $a_1$ is $J$-superficial,  $\{a_1,\dots,a_d\}$ is an $\m$-superficial sequence  and $e_1(\m)-e_1(J)=e_1(\m/(a_1))-e_1(J/(a_1)).$
 Now $A/(a_1)$ is a local ring of dimension $d-1$ and  $J/(a_1)$ is generated by a maximal $\m/(a_1)$-superficial sequence.
 We can then apply  Theorem \ref{th1} to get $$\sum_{j\ge 0}v_j(\m)=e_1(\m)-e_1(J)=e_1(\m/(a_1))-e_1(J/(a_1))\le \sum_{j\ge 0}v_j(\m/(a_1))\le \sum_{j\ge 0}v_j(\m).$$ This implies $$e_1(\m/(a_1))-e_1(J/(a_1))= \sum_{j\ge 0}v_j(\m/(a_1))$$ which, by the inductive assumption,  implies $A/(a_1)$ is Cohen-Macaulay. By Lemma \ref{nuovo}  $A$ is Cohen-Macaulay so that $e_1(J)=0$ and then $e_1(\m)=\sum_{j\ge 0}v_j(\m);$ this implies  depth $gr_{\m}(A)\ge d-1$ and the Theorem is proved.
\end{proof}

As a trivial consequence of the above result we have the following
\begin{corollary} Let $(A,\m)$ be a local ring of dimension $d\ge 1,$  $J$ the ideal generated by a maximal  $\m$-superficial sequence. If $e_1(J)\le 0,$ then $$e_1(\m)\le \sum_{j\ge 0} v_j(\m).$$  Moreover, the following conditions are equivalent:

\vskip 2mm \noindent
1. $e_1(\m)=\sum_{j\ge 0} v_j(\m).$

\vskip 2mm \noindent 2. $A$ is Cohen-Macaulay and
depth $gr_{\m}(A)\ge d-1.$ 
\end{corollary}

We notice that the condition $e_1(J)\le 0$ is satisfied if, for example, $A$ is Buchsbaum, see \cite{SV} Proposition 2.7. It is verified as well if depth $A\ge d-1;$ namely, if this is the case, given a  $J$-superficial sequence $a_1,\dots,a_{d-1}$, we have $e_1(J)=e_1(J/(a_1,\dots,a_{d-1})).$ Since $A/(a_1,\dots,a_{d-1})$ is one dimensional, by Lemma \ref{e1N} we get  $e_1(J/(a_1,\dots,a_{d-1}))\le 0.$ More in general Vasconcelos conjectured in \cite{Vas} that $e_1(J) <0 $ whenever $A$ is not Cohen-Macaulay. This is  proved for local integral domain essentially of finite type over a field.

\section{Application to the Sally module}

Given a primary ideal $\q$ in the local ring $(A,\m)$ and a minimal reduction $J$ of $\q,$ Vasconcelos introduced in \cite{Vas} the so-called {\bf Sally's} module $S_J(\q)$ of $\q$ with respect to $J.$ It is defined by the exact sequence 
$$0\to \q A[Jt]\to \q A[\q t]\to S_J(\q)=\oplus_{n\ge 1}\q^{n+1}/J^n\q\to 0.$$ We know that, if different from zero, $S_J(\q)$ is an $A[Jt]$-module of the same dimension $d$ as $A.$  

The Hilbert function of this graded module is $$H_{S_J(\q)}(n):=\la(q^{n+1}/J^n\q),$$ and its Poincare' series is $$P_{S_J(\q)}(z)=\sum_{n\ge 1}\la(q^{n+1}/J^n\q)z^n.$$  We write $e_i(S_J(\q))$ for the corresponding Hilbert coefficients.

In \cite{Vaz} it has been proved that if $A$ is Cohen-Macaulay
then $$e_0(S_J(\q))\le \sum_{j\ge 1}\la(\q^{j+1}/J\q^j)$$ and equality holds if and only if  depth $gr_{\q}(A) \ge d-1.$

In this section we extend the definition of the  Sally module to any good $\q$-filtration of an $A$-module $M$ and prove that the inequality above holds in this generality and without the assumption that  $M$ is Cohen-Macaulay. Further,  we study when equality holds in the special case of the $\q$-adic filtration.

\vskip 2mm Let $\M$ be a good $\q$-filtration of the $A$-module $M$  of dimension $d$ and let $J$ be the ideal generated by a maximal $\M$-superficial sequence for $\q.$ 

We let $$S_J(\M):=\oplus_{n\ge 1}(M_{n+1}/J^nM_1)$$ and call it the Sally's module of $\M$ with respect to $J.$ 
We remark that $S_J(\M)$ is not a graded module associated to a filtration, but if we consider the filtration $$\mathbb{E}:= M\supseteq M_1\supseteq JM_1\supseteq J^2M_1\supseteq \dots \supseteq J^nM_1\supseteq\dots$$ then $S_J(\M)$ is related to $gr_{\M}(M)$ and $gr_{\mathbb{E}}(M)$ by the following two short exact sequences:
$$0\to J^{n-1}M_1/J^nM_1\to M_n/J^nM_1\to M_n/J^{n-1}M_1\to 0$$ $$0 \to M_{n+1}/J^nM_1\to M_n/J^nM_1\to M_n/M_{n+1}\to 0.$$  
Since $M_{n }/J^{n-1}M_1= (S_J(\mathbb{M})(-1))_n, $ by standard facts it follows that
\begin{equation} {\label{depth}} \text{depth} \  gr _{\mathbb{M}}(M)  \ge \text{min}\{  \text{depth} \ S_J(\mathbb{M})  - 1, \text{depth} \ gr_{\mathbb{E}}(M)\}  \end{equation}

Moreover  we get 
	$$P_{S_J(\M)(-1)}(z)+P_{\mathbb{E }}(z)=P_{\M}(z)+P_{S_J(\M)}(z)$$ so that
\begin{equation} {\label{P(z)}}  (z-1)P_{S_J(\M)}(z)=P_{\M}(z)-P_{\mathbb{E}}(z)  \end{equation}
If $\dim S_J(\M)=d,$  this implies  that for every $i\ge 0$ we have 
\begin{equation} {\label{ei}} e_i(S_J(\M))=e_{i+1}(\M)-e_{i+1}(\mathbb{E})  \end{equation}

Now, by  Proposition \ref{EN1}, we know that
$$e_1(\mathbb{E})-e_1(\N)\ge e_0(\M)-h_0(\M),$$ so that  $$e_0(S_J(\M))=e_{1}(\M)-e_{1}(\mathbb{E})\le e_{1}(\M)-e_1(\N)-e_0(\M)+h_0(\M).$$ Notice that, for the special case of the $\q$-adic filtration on $A$ , this result  has been proved in \cite{C}.
\vskip 2mm

Now we can use  Theorem \ref{th1}, to get $$e_0(S_J(\M))\le \sum_{j\ge 0}v_j(\M) -e_0(\M)+h_0(\M)=\sum_{j\ge 1}v_j(\M).$$ Thus we have proved the following 
\begin{theorem} Let $\M$ be a good $\q$-filtration of the $A$-module $M$  of dimension $d$ and let $J$ be the ideal generated by a maximal $\M$-superficial sequence for $\q.$ If $ \dim S_J(\M)=d,$ then $$e_0(S_J(\M))\le \sum_{j\ge 1}v_j(\M)=\sum_{j\ge 1}\la(M_{j+1}/JM_j).$$
\end{theorem}

It is clear that in the above proof, if we have the equality
$e_0(S_J(\M))=\sum_{j\ge 1}v_j(\M),$ then $e_{1}(\M)-e_1(\N)=\sum_{j\ge 0}v_j(\M).$ Hence, in the case of the $\m$-adic filtration on $A$, we can apply Theorem \ref{th2} and we get that $A$ is Cohen-Macaulay and depth $gr_{\m}(A)\ge d-1.$ 

In this way we get the following result which completes  Theorem \ref{th2}.

\begin{theorem}\label{fin} Let $(A,\m)$ be a local ring of dimension $d\ge 1$ and $J$ the ideal generated by a maximal $\m$-superficial sequence. If $ \dim S_J(\m)=d,$ then   $e_0(S_J(\m))\le \sum_{j\ge 1}v_j(\m).$ Moreover the following conditions are equivalent:

\vskip 2mm
1. $e_0(S_J(\m))=\sum_{j\ge 1}v_j(\m)$ 

\vskip 2mm
2. $e_1(\m)-e_1(J)= \sum_{j\ge 0}v_j(\m)$

\vskip2mm
3. $A$ is Cohen-Macaulay and depth $gr_{\m}(A)\ge d-1.$
\end{theorem}

We finish the paper with the following Theorem which adds  some more equivalent conditions, involving the Sally module of $\M$,   to those of  Theorem \ref{e1=} in the case $M$ is Cohen-Macaulay.  The result extends   to a considerable extent a series of results proved in \cite{Vaz} for the special case of the $\q$-adic filtration on $A.$
\vskip 3mm
First we  remark that if  $M$ is Cohen-Macaulay,  then   $gr_{\mathbb{E}}(M) $ is Cohen-Macaulay with minimal multiplicity and  hence 
$$P_{ \mathbb{E} }(z)=  \frac{   h_0(\M) +  (e_0(\mathbb{M})- h_0(\M))z }{(1-z)^d}.$$ 
In particular  $ e_1(\mathbb{E}) = e_0(\M)-h_0(\M). $ 
\vskip 2mm
\noindent By using (\ref{ei}),  (\ref{depth}),  (\ref{P(z)})  we get
\vskip 2mm
1. If  $ \dim S_J(\M)=d,$ then 
    $e_0(S_J(\M))=e_{1}(\M)-e_0(\M)+h_0(\M) $ 
\vskip 2mm
2.    $ \text{depth} \  gr_{\mathbb{M}}(M)  \ge  \text{depth} \ S_J(\mathbb{M})  - 1$ 
\vskip 2mm
3. $  (z-1) P_{S_J(\mathbb{M})}(z) = P_{\M}(z) -   \frac{\lambda(M/M_1) +  (e_0(\mathbb{M})- (\lambda(M/M_1))z }{(1-z)^d}.$ 
\vskip 3mm
\noindent  If $\M$  is the $\q$-adic filtration on $M$ then it is easy to see that the assumption $ \dim S_J(\M)=d $ is equivalent to  $S_J(\M)\neq 0. $ In fact,  by (\ref{P(z)}), we have that $ \dim S_J(\M)=d $ if and only if $ e_{1}(\M) > e_1(\mathbb{E}) = e_0(\M)-h_0(\M). $ This is equivalent to $M_2 \neq J M_1 $ and hence $S_J(\M)\neq 0. $

\begin{theorem}\label {} Let $M$ be a Cohen-Macaulay  $A$-module of dimension $d\ge 1,$  $\q$ a primary ideal in $A,$ $\M$  the $\q$-adic filtration on $M$ and    $J$  the ideal generated by a maximal sequence of $\mathbb{M}$-superficial  elements for $\q.$ The following  conditions are equivalent : \vskip 2mm
1. $ e_0(S_J(\mathbb{M})) = \sum_{j\ge 1} v_j(\mathbb{M})  $ 
\vskip 2mm
2.   $ P_{S_J(\mathbb{M})} (z)  =   \frac{  \sum_{j\ge 1} v_{j}(\mathbb{M})  z^j }{(1-z)^r}$
\vskip 2mm
3.  $ S_J(\mathbb{M}) $ is Cohen-Macaulay 
\vskip 2mm
\noindent and each of them is equivalent to the equivalent conditions  of Theorem \ref{e1=}.
\end{theorem}
\begin{proof} We have  $e_0(S_J(\M))=e_{1}(\M)-e_0(\M)+h_0(\M). $  Hence, by Theorem \ref{e1=}, we get  $$e_0(S_J(\mathbb{M}))\le \sum_{j\ge 0} v_j(\mathbb{M})-e_0(\M)+h_0(\M) =\sum_{j\ge 1} v_j(\mathbb{M}).$$
By Theorem \ref{e1=}, the equality holds if and only if $ e_{1}(\mathbb{M})=\sum_{j\ge 0} v_j(\mathbb{M}) . $ Hence 1. is equivalent to {\it{ 1., 2., 3., 4., 5. }}    of  Theorem \ref{e1=}.     Because {\it{1.}}  is equivalent to 
$$P_{\M}(z)=\frac{\la(M/M_1)+\sum_{j\ge 0}(v_j(\M)-v_{j+1}(\M))z^{j+1}}{(1-z)^d} $$ and 
we know that $$  (z-1) P_{S_J(\mathbb{M})}(z) = P_{\M}(z) -   \frac{\lambda(M/M_1) +  (e_0(\mathbb{M})- (\lambda(M/M_1))z }{(1-z)^d},  $$   it is easy to see that {\it{ 1. }} is also equivalent to {\it{ 2.}}   Now,  since   $ \text{depth} \  gr_{\mathbb{M}}(M)  \ge  \text{depth} \ S_J(\mathbb{M})  - 1, $  we get  that {\it{ 3.}}  implies $ \text{depth} \  gr_{\mathbb{M}}(M)  \ge d-1 $ which is equivalent to {\it{ 1. }}

\noindent We have only to prove that  {\it{ 2.}}  implies {\it{ 3.}}
We may assume $S_J(\mathbb{M})$ of dimension $d $ and  recall that $S_J(\mathbb{M})$ is a $\mathbb{R }(J)=A[JT]$-module and we have $S_J(\mathbb{M})/JT  S_J(\mathbb{M})=\oplus_{n\ge1} M_{n+1}/JM_n.$  By {\it{ 2.}}  we deduce that   $ P_{S_J(\mathbb{M})} (z)=
\frac{1}{(1-z)^d} P_{S_J(\mathbb{M})/JT  S_J(\mathbb{M})}(z).$ Then $JT$ is generated by a regular sequence of lenght $d=$dim$S_J(\mathbb{M})$ and hence $ S_J(\mathbb{M}) $ is  Cohen-Macaulay.
\end{proof}

\providecommand{\bysame}{\leavevmode\hbox to3em{\hrulefill}\thinspace}

\end{document}